\documentclass[12pt,a4paper]{article}
\usepackage[latin1]{inputenc}
\usepackage{amsmath}
\usepackage{appendix}
\usepackage{multirow}
\usepackage{url}
\usepackage{graphicx}
\usepackage[all]{xy}
\usepackage{amsfonts}
\date{}
\usepackage{amsthm}
\usepackage{amssymb}
\theoremstyle{plain}

\newtheorem{theorem}{Theorem}[section]

\newtheorem{corollary}{Corollary}[section]
\newtheorem{example}{Example}[section]
\numberwithin{equation}{section}
\usepackage{fullpage}
\begin{document}

\begin{center} {\bf Modulo 2 congruences for partitions with initial repetitions}\end{center}
\begin{center}
 Darlison Nyirenda$^{1}$
and
Beaullah Mugwangwavari$^{2}$
 \vspace{0.5cm} \\
$^{1}$  The John Knopfmacher Centre for Applicable Analysis and Number Theory, 	University of the Witwatersrand, P.O. Wits 2050, Johannesburg, South Africa.\\
$^{2}$ School of Mathematics, University of the Witwatersrand, P. O. Wits 2050, Johannesburg, South Africa.\\
e-mails: darlison.nyirenda@wits.ac.za, 712040@students.wits.ac.za\\
\end{center}
\begin{abstract}
Motivated by Andrews' partitions with initial repetitions, we derive parity formulas for several functions for this class of partitions. In many cases, we present an infinite family of Ramanujan-like congruences modulo 2.
\end{abstract}

\section{Introduction}\label{sec1}
A partition of $n$ is a representation  $\lambda = \lambda_{1} + \lambda_{2} + \cdots + \lambda_{s}$ where $\lambda_i$'s are positive integers with $\lambda_1 \geq \lambda_{2} \geq \cdots \geq \lambda_{s}\geq 1$  and $\sum\limits_{i = 1}^{s}\lambda_i  = n$.  The summand $\lambda_i$ is called a part of $\lambda$, and the number of times this summand appears is called its multiplicity. Mostly, we use the multiplicity notation in which different parts appear with their multiplicities. We may write
$\lambda = (\mu_1^{m_1}, \mu_{2}^{m_2}, \ldots, \mu_{\ell}^{m_{\ell}})$ in which $m_i$ is the multiplicity of the part $\mu_i$ and $\mu_1 > \mu_2 > \cdots > \mu_{\ell}$.
\noindent At times, restrictions are made on $\mu_i$'s and $m_i$'s. In this case, we are said to have restricted partitions which induce restricted partition functions. For  more examples on this subject, one can refer to \cite{DN}. One interesting example is the case of partitions with initial repetitions, introduced by George Andrews \cite{initial}. Andrews defined a partition with initial $k$-repetitions as one in which if $j$ appears at least $k$ times, all positive integers less than $ j$ appear at least $k$ times. This definition entails that all parts greater than $j$ have their multiplicites strictly less than $k$. Using generating functions, Andrews proved that the number of partitions of $n$ with initial  $k$-repetitions is equal to the number of partitions of $n$ into parts with multiplicies not more than $2k - 1$.  A bijective proof of this identity was established by W. Keith \cite{keith}, and later, a simpler version thereof was given in \cite{nyirenda}. \\
\noindent Much of the work in this area has dwelt on partition identities and combinatorial bijections.  In this paper, we study a class of partitions with initial repetitions  with respect to parity. We derive modulo 2 congruences for partition functions associated with partitions in this category. We recall the following notation:\\
For $a, q \in \mathbb{C}$ and $n$ a positive integer,  $(a;q)_{n} = (1 - q)(1 - aq)(1 - aq^{2})\cdots (1 - aq^{n - 1})$  and 
$(a;q)_{0} = 1$.  If $\vert q\vert  < 1$, we have $(q;q)_{\infty} = \prod\limits_{n = 0}^{\infty}(1 - aq^{n})$ and so 
$$ (a;q)_{n} = \frac{(a;q)_{\infty}}{(aq^{n};q)_{\infty}}.$$
Throughout our discussion, we assume that $\vert q\vert  < 1$.  Some of the $q$-identities  which will be useful include:
 \begin{equation}\label{jacobi}
\sum\limits_{n = -\infty}^{\infty}z^{n}q^{n(n+1)/2} = \prod\limits_{n = 1}^{\infty}( 1 - q^{n})(1 + zq^{n})(1 + z^{-1}q^{n - 1}) 
\end{equation} for $z \neq 0$.
Other useful $q$-identities:
\begin{equation}\label{gauss}
\sum_{n = -\infty}^{\infty}(-1)^{n}q^{n^2} = \prod_{n = 1}^{\infty}\frac{1 - q^{n}}{1 + q^{n}},
\end{equation}
\begin{equation}\label{cauchy}
\sum_{n = 0}^{\infty}\frac{q^{n^2 - n} z^n}{(q;q)_{n}(z;q)_n} = \prod_{n = 0}^{\infty}(1 - zq^{n})^{-1},
\end{equation}
\begin{equation}\label{cube}
\sum_{n = 0}^{\infty}(-1)^{n}(2n + 1)q^{n(n+1)/2}= \prod_{n = 1}^{\infty}(1 - q^{n})^{3}.
\end{equation}
The results in this paper arise by examining the following identities of Rogers-Ramanujan type due to Slater (see \cite{slater}):
\begin{equation}\label{eq7}
\prod\limits_{n = 1}^{\infty} (1-q^{n}) \sum\limits_{n = 0}^{\infty} \frac{q^{n(n+1)}}{(q^{2};q^{2})_{n}} = \prod\limits_{n = 1}^{\infty}( 1 - q^{4n})(1 - q^{4n-1})(1 - q^{4n - 3}),
\end{equation}
\begin{equation}\label{eq18}
\prod\limits_{n = 1}^{\infty} (1-q^{n}) \sum\limits_{n = 0}^{\infty} \frac{q^{n^{2}}}{(q;q)_{n}} = \prod\limits_{n = 1}^{\infty}( 1 - q^{5n})(1 - q^{5n-2})(1 - q^{5n - 3}),
\end{equation}
\begin{equation}\label{eq23}
\prod\limits_{n = 1}^{\infty} \frac{(1-q^{2n})}{(1-q^{2n-1})} \sum\limits_{n = 0}^{\infty} \frac{(-1)^{n}q^{n^{2}}}{(q^{2};q^{2})_{n}} = \prod\limits_{n = 1}^{\infty}( 1 - q^{6n})(1 - q^{6n-2})(1 - q^{6n - 4}),
\end{equation}
\begin{equation}\label{eq27}
\prod\limits_{n = 1}^{\infty} (1-q^{2n}) \sum\limits_{n = 0}^{\infty} \frac{q^{2n(n+1)}(-q;q^{2})_{n}}{(q;q^{2})_{n+1}(q^{4};q^{4})_{n}} = \prod\limits_{n = 1}^{\infty}( 1 - q^{6n})(1 + q^{6n-1})(1 + q^{6n - 5}),
\end{equation}
\begin{equation}\label{eq29}
\prod\limits_{n = 1}^{\infty} \frac{(1-q^{2n})}{(1+q^{2n-1})} \sum\limits_{n = 0}^{\infty} \frac{q^{n^{2}}(-q;q^{2})_{n}}{(q;q)_{2n}} = \prod\limits_{n = 1}^{\infty}( 1 - q^{6n})(1 + q^{6n-2})(1 + q^{6n - 4}),
\end{equation}
\begin{equation}\label{eq31}
\prod\limits_{n = 1}^{\infty} (1-q^{2n}) \sum\limits_{n = 0}^{\infty} \frac{q^{2n(n+1)}}{(q^{2};q^{2})_{n}(-q;q)_{2n+1}} = \prod\limits_{n = 1}^{\infty}( 1 - q^{7n})(1 - q^{7n-1})(1 - q^{7n - 6}),
\end{equation}
\begin{equation}\label{eq32}
\prod\limits_{n = 1}^{\infty} (1-q^{2n}) \sum\limits_{n = 0}^{\infty} \frac{q^{2n(n+1)}}{(q^{2};q^{2})_{n}(-q;q)_{2n}} = \prod\limits_{n = 1}^{\infty}( 1 - q^{7n})(1 - q^{7n-2})(1 - q^{7n - 5}),
\end{equation}
\begin{equation}\label{eq33}
\prod\limits_{n = 1}^{\infty} (1-q^{2n}) \sum\limits_{n = 0}^{\infty} \frac{q^{2n^{2}}}{(q^{2};q^{2})_{n}(-q;q)_{2n}} = \prod\limits_{n = 1}^{\infty}( 1 - q^{7n})(1 - q^{7n-3})(1 - q^{7n - 4}),
\end{equation}
\begin{equation}\label{eq36}
\prod\limits_{n = 1}^{\infty} \frac{(1-q^{2n})}{(1+q^{2n+1})} \sum\limits_{n = 0}^{\infty} \frac{q^{n^{2}}(-q;q^{2})_{n}}{(q^{2};q^{2})_{n}} = \prod\limits_{n = 1}^{\infty}( 1 - q^{8n})(1 - q^{8n-3})(1 - q^{8n - 5}),
\end{equation}
\begin{equation}\label{eq50}
\prod\limits_{n = 1}^{\infty} (1-q^{n}) \sum\limits_{n = 0}^{\infty} \frac{q^{n(n+2)}(-q;q^{2})_{n}}{(q;q)_{2n+1}} = \prod\limits_{n = 1}^{\infty}( 1 - q^{12n})(1 - q^{12n-2})(1 - q^{12n - 10}).
\end{equation}
\noindent We present our results in Section 2.
\section{Modulo 2 congruences}
\noindent Unless otherwise specified, all congruence equations involving $q$-series are taken modulo 2. So the statement $a \equiv b$, where $a$ and $b$ are two expressions involving $q$-series, shall mean  $a \equiv b \pmod{2}$. We start by investigating a variation of Andrews' partitions with initial 2-repetitions as follows: \\
\noindent Let $c_{1}(n)$ be the number of partitions of $n$ in which  either
\begin{enumerate}
\item[(a)] all parts are distinct \\
or
\item[(b)] there is an odd repeated part $2j - 1$  and all positive integers less than $2j - 1$ appear as repeated parts, any part greater than $2j$ is distinct. 
\end{enumerate}
Then, we do get the following parity formula for $c_{1}(n)$.
\begin{theorem}\label{och}
For all $n\geq0$,
$$c_{1}(n) \equiv 1 \pmod{2} \iff n = j(j + 1)/2, j \geq 0.$$
\end{theorem}
\begin{proof}
\begin{align*}
\sum\limits_{n = 0}^{\infty}c_{1}(n)q^{n} & = \prod\limits_{j = 1}^{\infty}(1 + q^{j}) + \sum\limits_{n = 1}^{\infty} \frac{q^{1 + 1 + 2 + 2 + \ldots + (2n - 1) + (2n - 1)}}{(q;q)_{2n}}\prod\limits_{j = 2n + 1}^{\infty}(1 + q^{j}) \\
& = \sum\limits_{n = 0}^{\infty} \frac{q^{1 + 1 + 2 + 2 + \ldots + (2n - 1) + (2n - 1)}}{(q;q)_{2n}}\prod\limits_{j = 2n + 1}^{\infty}(1 + q^{j}) \\
                                      & = \sum\limits_{n = 0}^{\infty} \frac{q^{4n^2 - 2n}}{(q;q)_{2n}}\frac{\prod\limits_{j = 1}^{\infty}(1 + q^{j})}{\prod\limits_{j = 1}^{2n}(1 + q^{j})}\\
                                      & = (-q;q)_{\infty}\sum\limits_{n = 0}^{\infty} \frac{q^{4n^2 - 2n}}{(q;q)_{2n}(-q;q)_{2n}} \\
                                      & = (-q;q)_{\infty}\sum\limits_{n = 0}^{\infty} \frac{q^{4n^2 - 2n}}{(q^{2};q^{2})_{2n}} \\
                                      & = (-q;q)_{\infty}\sum\limits_{n = 0}^{\infty} \frac{q^{4n^2 - 2n}}{(q^{4};q^{4})_{n}(q^{2};q^{4})_{n}}\\
                                      & = (-q;q)_{\infty}\prod\limits_{n = 0}^{\infty}\frac{1}{1 - q^{2 + 4n}}\,\,\,\,\,\, (z = q^{2}, q: = q^{4}\,\,\,\, \text{in}\,\, \eqref{cauchy}) \\
                                      & = \frac{(-q;q^{2})_{\infty}(-q^{2};q^{2})_{\infty}}{(-q;q^{2})_{\infty}(q;q^{2})_{\infty}} \\
                                      & = \frac{(q^{4};q^{4})_{\infty}}{(q;q)_{\infty}} \end{align*}
\begin{align*}
                                      &  \equiv (q;q)^{3}\\
                                      & \equiv \sum_{n  = 0}^{\infty}q^{n(n + 1)/2}.
\end{align*}
\end{proof}

\noindent Let $c_{2}(n)$ be the number of partitions of $n$ in which there exists $j \geq 1$ such that $j$ appears exactly $j$ times and it is the only part less than $2j + 1$, even parts $\geq 2j + 2$ are distinct, odd parts $\geq 2j + 1$ appear unrestricted.  Then we have the following theorem.
\begin{theorem}
For all $n\geq 0$,
$$c_{2}(5n + 2) \equiv 0 \pmod{2}.$$
\end{theorem}
\begin{proof}
\begin{align*}
\sum_{n = 0}^{\infty} c_{2}(n)q^{n} & = \sum_{n = 1}^{\infty} \dfrac{q^{n^{2}}(-q^{2n+2};q^{2})_{\infty}}{(q^{2n+1};q^{2})_{\infty}} \\
& =   \sum_{n = 0}^{\infty} \dfrac{q^{n^{2}}(-q^{2n+2};q^{2})_{\infty}}{(q^{2n+1};q^{2})_{\infty}} -\frac{(-q^2;q^2)_{\infty}}{(q;q^{2})_{\infty}} \\
& \equiv \sum_{n = 0}^{\infty} \dfrac{q^{n^{2}}(q^{2n+2};q^{2})_{\infty}}{(q^{2n+1};q^{2})_{\infty}} + (q^{2};q^{2})_{\infty}(-q;q)_{\infty} \\
& = \sum_{n = 0}^{\infty} \dfrac{q^{n^{2}}(-q;q^{2})_{n}(q^{2};q^{2})_{\infty}}{(-q;q^{2})_{\infty}(q^{2};q^{2})_{n}}  + (q;q)_{\infty}^{3} \\
& = \dfrac{(q^{2};q^{2})_{\infty}}{(-q;q^{2})_{\infty}} \sum_{n = 0}^{\infty} \dfrac{q^{n^{2}}(-q;q^{2})_{n}}{(q^{2};q^{2})_{n}} + \sum\limits_{n = 0}^{\infty}(-1)^{n}(2n + 1)q^{n(n + 1)/2} \\
& \equiv \prod_{n=1}^{\infty} \left(1-q^{8n-3}\right)\left(1-q^{8n-5}\right)\left(1-q^{8n}\right)  + \sum\limits_{n = 0}^{\infty}q^{n(n + 1)/2}  \\
& \quad \quad (\text{by}\,\, \eqref{eq36} \,\,\text{and}\,\,\eqref{cube})\\
& \equiv \sum_{n = -\infty}^{\infty} q^{4n^{2}+n}  + \sum\limits_{n = 0}^{\infty}q^{n(n + 1)/2}\,\,\,\,(\text{by} \,\,\,\,\eqref{jacobi}).
\end{align*}
Since there is no integer $n$ such that $4n^{2} + n \equiv 2 \pmod{5}$ or $n(n + 1)/2 \equiv 2 \pmod{5}$, it must follow that
$c_{2}(5n + 2) \equiv 0 \pmod{2}$. 
\end{proof}
For instance, there are two partitions of 7  enumerated by $c_{2}(7)$. These are: $(6,1)$ and $(3^{2}, 1)$. Thus $c_{2}(7)\equiv 0 \pmod{2}$.\\ \par
\noindent Let $c_{3}(n)$ be the number of partitions of $n$ in which, there is a positive integer $j$ such that 1 appears with multiplicities $j^{2}$  or $j^{2} + 1$, odd parts $> 1$ are distinct, all even parts are distinct and those $> 2j$ are at least $4j + 4$ in size and divisible by 4, no even integer in the set $\{2j + 2, 2j + 4, \ldots, 4j + 2\}$ appears as a part. Then
\begin{align*}
\sum_{n = 0}^{\infty}c_{3}(n)q^{n} & = \sum_{n = 1}^{\infty}q^{1 + 1 + \cdots + 1 (n^{2}\,\text{times})}(1 + q)(-q^{3};q^{2})_{\infty}(-q^{2};q^{2})_{n} (-q^{4n + 4};q^{4})_{\infty}\\
& \equiv \sum_{n = 1}^{\infty}q^{n^{2}}(1 + q)(-q^{3};q^{2})_{\infty}(-q^{2};q^{2})_{n}(-q^{2n + 2};q^{2})_{\infty} (-q^{2n + 2};q^{2})_{\infty} \\
& =  \sum_{n = 1}^{\infty}q^{n^{2}}  (-q;q^{2})_{\infty}(-q^{2};q^{2})_{\infty}(-q^{2n + 2};q^{2})_{\infty}\\
& =  \sum_{n = 1}^{\infty}q^{n^{2}}(-q;q)_{\infty}(-q^{2n + 2};q^{2})_{\infty}\\
& = (-q;q)_{\infty}\sum_{n = 1}^{\infty}q^{n^{2}}\frac{(-q^{2};q^{2})_{\infty}}{q^{2};q^{2})_{n}}\\
                                                  & \equiv \frac{ (q^{2};q^{2})_{\infty}  }{(q;q^{2})_{\infty}} \sum_{n = 1}^{\infty}\frac{q^{n^{2}}}{(q^{2};q^{2})_{n}}  \\
& \equiv  \frac{(q^{2};q^{2})_{\infty}}{(q;q^{2})_{\infty}}\sum_{n = 1}^{\infty}\frac{(-1)^{n}q^{n^{2}}}{(q^{2};q^{2})_{n}} \\
& \equiv  \frac{(q^{2};q^{2})_{\infty}}{(q;q^{2})_{\infty}}\left(\sum_{n = 0}^{\infty}\frac{(-1)^{n}q^{n^{2}}}{(q^{2};q^{2})_{n}}  - 1\right) \\
& =  \frac{(q^{2};q^{2})_{\infty}}{(q;q^{2})_{\infty}}\sum_{n = 0}^{\infty}\frac{(-1)^{n}q^{n^{2}}}{(q^{2};q^{2})_{n}}  - \frac{(q^{2};q^{2})_{\infty}}{(q;q^{2})_{\infty}} \\
& \equiv  (q^{2};q^{6})_{\infty} (q^{4};q^{6})_{\infty} (q^{6};q^{6})_{\infty} +  \frac{(q^{2};q^{2})_{\infty}}{(q;q^{2})_{\infty}}  \,\,\,(\text{by}\,\,\,\, \eqref{eq23}) \\
& = \sum_{n = -\infty}^{\infty}q^{n(3n + 1)}  +  \sum_{n = 0}^{\infty}q^{n(n + 1)/2}\,\,\,\,\,(\text{by}\,\,\,\,\eqref{jacobi}).
\end{align*}
Note that, for all $n \in \mathbb{Z}$, we have $n(3n + 1) \equiv  0,2,4 \pmod{5}$  and for all $n \in \mathbb{Z}_{\geq 0}$, $n(n+1)/2  \equiv 0,1,3 \pmod{5}$. More specifically, we have
$n(3n + 1) \equiv  0 \pmod{5}  \iff  n \equiv 0,3 \pmod{5}$,\,\,\,\, $n(3n + 1) \equiv  2 \pmod{5} \iff n \equiv 4 \pmod{5}$, \,\,\,\,$n(3n + 1) \equiv  4 \pmod{5} \iff n \equiv 1,2 \pmod{5}$,\,\,\,\, $n(n + 1)/2 \equiv  0 \pmod{5} \iff n \equiv 0, 4 \pmod{5}$,\,\,\, $n(n + 1)/2 \equiv  1 \pmod{5} \iff \,\,\, n \equiv 1, 3 \pmod{5}$  and  $n(n + 1)/2 \equiv  3 \pmod{5} \,\,\iff\,\, n \equiv 2 \pmod{5}$. \\
Hence
$$ \sum_{n = 0}^{\infty}c_{3}(5n + 1)q^{5n + 1} \equiv   \sum_{n \geq 1, n \equiv 1,3 \pmod{5}}q^{n(n + 1)/2}\pmod{2},$$
$$ \sum_{n = 0}^{\infty}c_{3}(5n + 2)q^{5n + 2} \equiv   \sum_{n \equiv 4 \pmod{5}}q^{n(3n + 1)} = \sum\limits_{n = -\infty}^{\infty} q^{(5n + 4)(15n + 13)}\pmod{2},$$
$$ \sum_{n = 0}^{\infty}c_{3}(5n + 3)q^{5n + 3} \equiv   \sum_{n > 1, n \equiv 2 \pmod{5}}q^{n(n + 1)/2} = \sum\limits_{n = 0}^{\infty} q^{(5n + 2)(5n + 3)/2}$$
and 
$$ \sum_{n = 0}^{\infty}c_{3}(5n + 4)q^{5n + 4} \equiv   \sum_{n \equiv 1, 2 \pmod{5}}q^{n(3n + 1)}$$
so that we have the following.
\begin{theorem}
For all $n\geq 0$,
\begin{equation*}
c_{3}(5n + 1) \equiv 1 \pmod{2} \,\,\,\,\text{iff}\,\,\,\, n = \frac{j(j + 1) - 2}{10},\,\,\, j \geq 1\,\,\,\text{and}\,\,\, j \equiv 1, 3 \pmod{5},
\end{equation*}
\begin{equation*}
c_{3}(5n + 2) \equiv 1\pmod{2}\,\,\,\,\text{iff}\,\,\,\, n = \frac{(5j + 4)(15j + 13) - 2}{5},\,\,\,j \in \mathbb{Z}.
\end{equation*}
\begin{equation*}
c_{3}(5n + 3) \equiv 1\pmod{2}\,\,\,\,\text{iff}\,\,\,\, n = \frac{(5j + 2)(5j+ 3) - 6}{10},\,\, j \geq 0,
\end{equation*}
\begin{equation*}
c_{3}(5n + 4) \equiv 1\pmod{2}\,\,\,\,\text{iff}\,\,\,\, n = \frac{j(3j + 1) - 4}{5},\,\, j \in \mathbb{Z}\,\,\,\text{and}\,\,\, j \equiv 1, 2\pmod{5}.
\end{equation*}
\end{theorem}
\noindent Recall that
$$\sum_{n  = 0 }^{\infty}c_{3}(n)q^{n} \equiv  \sum_{n = -\infty}^{\infty}q^{n(3n + 1)}  +  \sum_{n = 0}^{\infty}q^{n(n + 1)/2}.$$
\noindent Observe that none of the exponents $n(3n + 1)$ or  $n(n + 1)/2$ is congruent to 5, 7, 9 modulo 11. Thus, we have:

\begin{theorem}\label{same0}
For all $n\geq 0$,
\begin{equation*}
c_{3}(11n + 5) \equiv 0 \pmod{2},
\end{equation*}
\begin{equation*}
c_{3}(11n + 7) \equiv 0 \pmod{2},
\end{equation*}
\begin{equation*}
c_{3}(11n + 9) \equiv 0 \pmod{2}.
\end{equation*}
\end{theorem}

\noindent Let $c_{4}(n)$ denote the number of partitions of $n$ in which  either
\begin{enumerate}
\item[(a)] all parts are even and distinct \\
or
\item[(b)] there is an even part $2j$ which appears twice, all positive even integers $< 2j$ appear twice, any even part larger than $2j$ is actually $\geq 4j + 2$ and distinct, odd parts are distinct and at most $2j - 1$ in part size. 
\end{enumerate}
Then we have the following theorem.
\begin{theorem}\label{part32}
For all $n\geq 0$,
$$ c_{4}(n)
\equiv \begin{cases}
1 \pmod{2}, & n = (7j^2 + 3j)/2, j \in \mathbb{Z};\\\\
0 \pmod{2},& \text{otherwise}.
\end{cases}
$$
\end{theorem}
\begin{proof}
Note that
\begin{align*}
\sum_{n \geq 0} c_{4}(n)q^{n} & = (-q^{2};q^{2})_{\infty} + \sum_{n \geq 1} q^{2 + 2 + 4 + 4 + \cdots + 2n + 2n}\prod\limits_{i=1}^{n}(1 + q^{2i - 1})\prod\limits_{j=1}^{\infty}(1 + q^{4n + 2j}) \\
                              & = \sum_{n\geq 0}q^{2n(n + 1)}(-q;q^{2})_{n}(-q^{4n + 2};q^{2})_{\infty}\\
                              & \equiv \sum_{n\geq 0}q^{2n(n + 1)}(q;q^{2})_{n}(q^{4n + 2};q^{2})_{\infty}\\
                              & = \sum_{n\geq 0}q^{2n(n + 1)}(q;q^{2})_{n}\frac{(q^{2};q^{2})_{2n}}{(q^{2};q^{2})_{2n} }(q^{4n + 2};q^{2})_{\infty}\\
                              & = \sum_{n\geq 0}q^{2n(n + 1)}(q;q^{2})_{n}\frac{(q^{2};q^{2})_{\infty}}{(q^{2};q^{2})_{2n} }\\
                              & = (q^{2};q^{2})_{\infty}\sum_{n \geq 0} q^{2n(n + 1)}\frac{(q;q^{2})_{n}}{(q;q)_{2n}(-q;q)_{2n}} \\
                              & = (q^{2};q^{2})_{\infty}\sum_{n \geq 0} \frac{q^{2n(n + 1)}}{(q^{2};q^{2})_{n}(-q;q)_{2n}} \\
                              & \equiv \prod_{n \geq 1}(1 - q^{7n})(1 + q^{7n - 2})(1 + q^{7n - 5})\quad \,\,\,(\text{by}\,\,\,\, \eqref{eq32}) \\
                              & \equiv \sum_{n = -\infty}^{\infty}q^{ (7n^2 + 3n)/2}.                                                    
\end{align*}
\end{proof}
\noindent Let $c_{5}(n)$ denote the number of partitions of $n$ in which either
\begin{enumerate}
\item[(a)] all parts are even and distinct \\
 or
\item[(b)] odd parts appear twice or thrice and are gap-free, even parts are distinct and the smallest even part is $\geq$ 2(the largest odd part) + 4. 
\end{enumerate}
Then parity of $c_{5}(n)$ can be deduced from the following.
\begin{theorem}\label{part33}
For all $n\geq 0$,
$$ c_{5}(n)
\equiv \begin{cases}
1\pmod{2}, & n = (7j^2 + j)/2, j \in \mathbb{Z};\\\\
0 \pmod{2},& \text{otherwise}.
\end{cases}
$$
\end{theorem}
\begin{proof}
It is clear that $$\sum\limits_{n \geq 0}c_{5}(n)q^{n} = (-q^{2};q^{2})_{\infty} + \sum\limits_{n \geq 1} q^{1 + 1 + 3 + 3 + \ldots + 2n - 1}(-q;q^{2})_{n}(-q^{4n + 2}; q^{2})_{\infty},$$ and by a similar manipulation as in Theorem \ref{part32}, we have
\begin{align*}
\sum\limits_{n \geq 0}c_{5}(n)q^{n} & \equiv (q^{2};q^{2})_{\infty}\sum_{n \geq 0}\frac{q^{2n^{2}}}{(q^{2};q^{2})_{n}(-q;q)_{2n}}  \\
                                    & \equiv \prod_{n \geq 1}(1 - q^{7n})(1 + q^{7n - 3})(1 + q^{7n - 4})\quad (\text{by}\,\,\,\, \eqref{eq33})       \\
                                    & \equiv \sum\limits_{n = -\infty}^{\infty} q^{(7n^2 + n)/2}.                                            
\end{align*}
\end{proof}
\noindent Let $c_{6}(n)$ denote the number of partitions of $n$ in which either
\begin{enumerate}
\item[(a)] 1 is the only odd integer that may appear and even parts are distinct \\
or
\item[(b)] there is an even part $2j$ that appears twice, all positive even integers $< 2j$ appear twice, any even part larger than $2j$ is actually $\geq 4j + 2$ and distinct, odd parts are $\leq 2j + 1$  and those $\leq 2j - 1$ are distinct. 
\end{enumerate}
We have:
\begin{theorem}\label{part31}
For all $n\geq 0$,
$$ c_{6}(n)
\equiv \begin{cases}
1\pmod{2}, & n = (7j^2 + 5j)/2, j\in \mathbb{Z};\\\\
0 \pmod{2},& \text{otherwise}.
\end{cases}
$$
\end{theorem}
\begin{proof}
The generating function for the partition function in question is 
$$\sum_{n \geq 0}c_{6}(n)q^{n} = \frac{(-q^{2};q^{2})_{\infty}}{1 - q} + \sum_{n \geq 1} q^{2 + 2 + 4 + 4 + \ldots 2n + 2n}\frac{(-q;q^{2})_{n}}{1 - q^{2n + 1}} (-q^{4n + 2};q^{2})_{\infty}.$$
However, by a similar manipulation as in the proof of Theorem \ref{part32}, we find that
\begin{align*}
\sum_{n \geq 0}c_{6}(n)q^{n} &  \equiv (q^{2};q^{2})_{\infty}\sum_{n \geq 0}\frac{q^{2n(n + 1)}}{(q^{2};q^{2})_{n}(-q;q)_{2n + 1}} \pmod{2}\\
                             & \equiv \prod_{n \geq 1}(1 - q^{7n})(1 + q^{7n - 1})(1 + q^{7n - 6}) \quad (\text{by}\,\,\,\, \eqref{eq31}) \\
                             & \equiv \sum_{n = -\infty}^{\infty} q^{(7n^2 + 5n)/2}.
\end{align*}
\end{proof}
\noindent Let $c_{7}(n)$ denote the number of partitions of $n$ in which either
\begin{enumerate}
\item[(a)] even parts are $\equiv 2 \pmod{4}$ and distinct, and 1 is the only odd integer that may appear\\
 or
\item[(b)] the largest even part $2j$ appears exactly twice if $2j \equiv 0 \pmod{4}$, and appears twice or thrice if  $2j \equiv 2 \pmod{4}$, all positive even integers  $< 2j$ and $\equiv 0 \pmod{4}$ are repeated exactly twice, and those  $< 2j$ and $\equiv 2 \pmod{4}$ appear twice or thrice, any even part larger than $2j$ that is divisible by 4 is actually at least $4j + 4$ in part size and distinct, even parts that are $> 2j$  and $\equiv 2 \pmod{4}$ are distinct and $2j + 1$  is the only odd integer that may appear. 
\end{enumerate}
The following result follows:
\begin{theorem}\label{part27}
For all $n\geq 0$,
$$ c_{7}(n)
\equiv \begin{cases}
1\pmod{2}, & n = 3j^2 + 2j, j \in \mathbb{Z};\\\\
0\pmod{2}, & \text{otherwise}.
\end{cases}
$$
\end{theorem}
\begin{proof}
Note that 
\begin{align*}
\sum\limits_{n\geq 0}c_{7}(n)q^{n} & = \frac{(-q^{2};q^{4})_{\infty}}{1 - q} + (-q^{2};q^{4})_{\infty}\sum_{n \geq 1}\frac{q^{2(2 + 4 + 6 + \cdots + 2n)}}{1 - q^{2n + 1}}(-q^{4n + 4};q^{4})_{\infty}\\
                            & \equiv (q^{2};q^{4})_{\infty}\sum_{n \geq 0}\frac{q^{2n(n + 1)}}{1 - q^{2n + 1}}(q^{4n + 4};q^{4})_{\infty} \\
                            & = \sum_{n \geq 0}\frac{q^{2n(n + 1)}}{1 - q^{2n + 1}}(q^{2};q^{4})_{\infty}\frac{(q^{4};q^{4})_{\infty}}{(q^{4};q^{4})_{n}} \\
                            & = \sum_{n \geq 0}\frac{q^{2n(n + 1)}}{(1 - q^{2n + 1})(q^{4};q^{4})_{n}}(q^{2};q^{2})_{\infty} \\
                            & = \sum_{n \geq 0}\frac{q^{2n(n + 1)}}{(1 - q^{2n + 1})}\frac{(q;q^{2})_{n}}{(q;q^{2})_{n}(q^{4};q^{4})_{n}}(q^{2};q^{2})_{\infty}\\
                            & \equiv (q^{2};q^{2})_{\infty}\sum_{n \geq 0}\frac{q^{2n(n + 1)}}{(q;q^{2})_{n + 1} }\frac{(-q;q^{2})_{n}}{(q^{4};q^{4})_{n}}\\
                            & = \prod_{n \geq 1}(1 - q^{6n})(1 + q^{6n - 5})(1 + q^{6n - 1}) \,\,\,(\text{by}\,\,\,\, \eqref{eq27}) \\
                            & \equiv \sum_{n = -\infty}^{\infty}q^{3n^2 + 2n}.
\end{align*}
\end{proof}
\noindent Let $c_{8}(n)$ denote the number of partitions of $n$ in which either
\begin{enumerate}
\item[(a)] all parts are even and distinct\\
 or
\item[(b)] there is the largest odd part $2j - 1$ which appears once, all positive odd integers $\leq j$ appear once or twice , all positive odd integers $> j$ appear once, even parts  
$\leq j$ are distinct, and those $> j$ are distinct and actually $\geq 2j + 2$ in size.
\end{enumerate}
\noindent We obtain the following theorem.
\begin{theorem}\label{part18}
For all $n\geq 0$,
\begin{equation*}
c_{8}(49n + r) \equiv 0 \pmod{2},
\end{equation*}
where $r = 6,20,27,34,41,48.$
\end{theorem}
\begin{proof}
It is not difficult to see that
\begin{align*}
\sum_{n\geq 0}c_{8}(n)q^{n}  & = (-q^{2};q^{2})_{\infty} + \sum_{n \geq 1}q^{1 + 3 + 5 + \ldots + 2n - 1}(1 + q)(1 + q^{2})\ldots (1 + q^{n})(-q^{2n + 2};q^{2})_{\infty} \\
                       & = \sum_{n \geq 0}q^{n^{2}}(-q;q)_{n}(-q^{2n + 2};q^{2})_{\infty} \\
                       & = \sum_{n \geq 0}q^{n^{2}}\frac{(q;q)_n}{(q;q)_n}(-q;q)_{n}(-q^{2n + 2};q^{2})_{\infty} \\
                       & \equiv \sum_{n \geq 0}q^{n^{2}}\frac{(q^{2};q^{2})_n}{(q;q)_n}(q^{2n + 2};q^{2})_{\infty} \\
                       & = (q^{2};q^{2})_{\infty}\sum_{n \geq 0}\frac{q^{n^{2}}}{(q;q)_n}\\
                       & \equiv (q;q)_{\infty} \prod_{n \geq 1}(1 - q^{5n})(1 + q^{5n - 2})(1 + q^{5n - 3}) \,\,\,(\text{by}\,\,\,\, \eqref{eq18}) \\
                      & = \sum_{n = -\infty}^{\infty}q^{n(3n + 1)/2} \sum_{n = -\infty}^{\infty}q^{n(5n + 1)/2} \\
                      & = \sum_{n = -\infty}^{\infty}\left(q^{n(6n + 1)}  + q^{(2n + 1)(3n + 2)}\right) \sum_{n = -\infty}^{\infty}\left(q^{n(10n + 1)} + q^{(2n + 1)(5n + 3)}\right).  
\end{align*}
\noindent Since the exponents in $ \sum\limits_{n =- \infty}^{\infty}\left(q^{n(6n + 1)}  + q^{(2n + 1)(3n + 2)}\right)$ are congruent to \\
$$ 0, 1, 2, 5, 7, 8, 12, 14, 15, 19, 21, 22, 26,28,29, 33,35,36,40,42,43,47$$ modulo 49 and  the exponents in $\sum\limits_{n = -\infty}^{\infty}\left(q^{n(10n + 1)} + q^{(2n + 1)(5n + 3)}\right)$ are  congruent to 
$$ 0 ,2, 3, 7, 9,  10, 11, 14, 16,  17, 21, 23, 24, 28, 30, 31, 35, 37, 38, 42, 44, 45$$ modulo 49, it follows that the product \\
$\sum\limits_{n = -\infty}^{\infty}\left(q^{n(6n + 1)}  + q^{(2n + 1)(3n + 2)}\right) \sum\limits_{n = -\infty}^{\infty}\left(q^{n(10n + 1)} + q^{(2n + 1)(5n + 3)}\right) $ has no exponent congruent to  6, 20, 27, 34, 41, 48 modulo 49. Thus
\begin{equation*}
\sum_{n \geq 0}^{\infty}c_{8}(49n + r)q^{49n + r} \equiv 0 \pmod{2} 
\end{equation*}
where $r = 6, 20, 27, 34, 41, 48$.
\end{proof}
\noindent Let $c_{9}(n)$ be the number of partitions of $n$ in which either
\begin{enumerate}
\item[(a)] all parts are distinct and greater than 1 \\
or
\item[(b)]  there exists $j \geq 2$ such that 1 appears exactly $j^{2}$ times and  parts $> 1$ are at least  $j + 1$ in size and distinct.
\end{enumerate}
Then, we have:
\begin{theorem}\label{part27}
For all $n\geq 0$,
$$ c_{9}(n)
\equiv \begin{cases}
1\pmod{2}, & n = (5j^{2}+j)/2, j \in \mathbb{Z};\\\\
0\pmod{2}, & \text{otherwise}.
\end{cases}
$$
\end{theorem}
\begin{proof}
We have:
\begin{align*}
\sum_{n = 0}^{\infty} c_{9}(n)q^{n} & = (-q^{2};q)_{\infty} + \sum_{n = 2}^{\infty} q^{n^{2}}(-q^{n+1};q)_{\infty} \\
& \equiv (1+2q)(-q^{2};q)_{\infty} + \sum_{n = 2}^{\infty} q^{n^{2}}(-q^{n+1};q)_{\infty} \\
& = (1+q)(-q^{2};q)_{\infty} + q(-q^{2};q)_{\infty} + \sum_{n = 2}^{\infty} q^{n^{2}}(-q^{n+1};q)_{\infty} \\
& = (-q;q)_{\infty} + q(-q^{2};q)_{\infty} + \sum_{n = 2}^{\infty} q^{n^{2}}(-q^{n+1};q)_{\infty} \\
& = \sum_{n = 0}^{\infty} q^{n^{2}}(-q^{n+1};q)_{\infty} \\
& \equiv \sum_{n = 0}^{\infty} q^{n^{2}}(q^{n+1};q)_{\infty} \\
& = \sum_{n = 0}^{\infty} \dfrac{q^{n^{2}}(q;q)_{\infty}}{(q;q)_{n}} \\
& = (q;q)_{\infty} \sum_{n = 0}^{\infty} \dfrac{q^{n^{2}}}{(q;q)_{n}} \\
& \equiv \prod_{n=1}^{\infty} \left(1 + q^{5n-2}\right)\left(1 +q^{5n-3}\right)\left(1-q^{5n}\right) \,\,\,(\text{by}\,\,\,\, \eqref{eq18})\\
& \equiv \sum_{n = -\infty}^{\infty} q^{\frac{5n^{2}+n}{2}}.
\end{align*}
\end{proof}
\begin{example}
Consider $n = 11$.
\end{example}
The $c_{9}(11)$-parititions are:
$$  11, (9,2), (8,3), (7,4), (7,1^{4}), (6,5), (6,3,2), (5,4,2), (4,3,1^{4}) $$
and so
 $c_{9}(11) =9 \equiv 1 \pmod{2}$ . Indeed this is true since $11 =\frac{ 5(2)^{2} + 2}{2} ( j = 2)$ in the theorem.

\noindent Let $c_{10}(n)$ be the number of partitions of $n$ in which either
\begin{enumerate}
\item[(a)] all parts are distinct  \\
or
\item[(b)] there exists $j \geq 1$ such that all positive odd integers $\leq j$ appear twice or thrice and other odd parts are distinct, all positive even integers $\leq j$ appear twice, even parts $>2j$ are distinct and no even integer in the interval $[j + 1, 2j]$ appears.
\end{enumerate}
Then, we have:
\begin{theorem}\label{part27}
For all $n\geq 0$,
$$ c_{10}(n)
\equiv \begin{cases}
1\pmod{2}, & n = 2j^{2}+ j, j \in \mathbb{Z};\\\\
0\pmod{2} & \text{otherwise}.
\end{cases}
$$
\end{theorem}

\begin{proof}
The generating function for $c_{10}(n)$ is
\begin{align*}
\sum_{n = 0}^{\infty} c_{10}(n)q^{n} & = \sum_{n = 0}^{\infty} q^{n(n+1)}(-q^{2n+2};q^{2})_{\infty}(-q;q^{2})_{\infty} \\
& \equiv (q;q^{2})_{\infty} \sum_{n = 0}^{\infty} q^{n(n+1)}(q^{2n+2};q^{2})_{\infty} \,\, \pmod{2} \\
& = (q;q^{2})_{\infty} \sum_{n = 0}^{\infty} \dfrac{q^{n(n+1)}(q^{2};q^{2})_{\infty}}{(q^{2};q^{2})_{n}} \\
& = (q;q)_{\infty} \sum_{n = 0}^{\infty} \dfrac{q^{n(n+1)}}{(q^{2};q^{2})_{n}} \\
& = \prod_{n=1}^{\infty} \left(1-q^{4n-1}\right)\left(1-q^{4n-3}\right)\left(1-q^{4n}\right)  \,\,\,(\text{by}\,\,\,\, \eqref{eq7}) \\
& \equiv \sum_{n = -\infty}^{\infty} q^{2n^{2}+n}.
\end{align*}
 \end{proof}
\noindent Let $\tilde{c}_{11}(n)$ be the number of partitions of $n$ in which there is $j\geq 1$ such that 1 appears exactly $j^{2}$, odd parts $> 1$ appear unrestricted, even parts are $>2j$ and distinct. Define $c_{11}(n)$ as follows:
$c_{11}(n) = \sum\limits_{i = 0}^{n}\tilde{c}_{11}(i)$. Then 
$$ \sum_{n = 0}^{\infty}\tilde{c}_{11}(n)q^{n} = \sum_{n = 1}^{\infty} \dfrac{q^{n^{2}}(-q^{2n+2};q^{2})_{\infty}}{(q^{3};q^{2})_{\infty}}$$ and 
\begin{align*}
\sum_{n = 0}^{\infty} c_{11}(n)q^{n} & =\frac{1}{1 - q} \sum_{n = 0}^{\infty}\tilde{c}_{11}(n)q^{n} \\
& = \sum_{n = 1}^{\infty} \dfrac{q^{n^{2}}(-q^{2n+2};q^{2})_{\infty}}{(q;q^{2})_{\infty}} \\
& \equiv \sum_{n = 1}^{\infty} \dfrac{q^{n^{2}}(q^{2};q^{2})_{\infty}}{(q^{2};q^{2})_{n}(q;q^{2})_{\infty}} \,\, \pmod{2} \\
& = \sum_{n = 1}^{\infty} \dfrac{q^{n^{2}}(q^{2};q^{2})_{\infty}(q;q^{2})_{n}}{(q^{2};q^{2})_{n}(q;q^{2})_{\infty}(q;q^{2})_{n}} \\
& \equiv \sum_{n = 1}^{\infty} \dfrac{q^{n^{2}}(q^{2};q^{2})_{\infty}(-q;q^{2})_{n}}{(q^{2};q^{2})_{n}(-q;q^{2})_{\infty}(q;q^{2})_{n}} \\
& \equiv  \frac{  (q^{2};q^{2})_{\infty} }{  (-q;q^{2})_{\infty} }\sum_{n = 0}^{\infty} \dfrac{q^{n^{2}}(-q;q^{2})_{n}}{(q^{2};q^{2})_{n}(q;q^{2})_{n}}  - \frac{(-q^{2};q^{2})_{\infty}}{(q;q^{2})_{\infty}} \\
&  \equiv  \frac{  (q^{2};q^{2})_{\infty} }{  (-q;q^{2})_{\infty} } \sum_{n = 0}^{\infty} \dfrac{q^{n^{2}}(-q;q^{2})_{n}}{(q;q)_{2n}}  + (q;q)_{\infty}^{3} \\
& \equiv \prod_{n=1}^{\infty} \left(1 + q^{6n-2}\right)\left(1 + q^{6n-4}\right)\left(1-q^{6n}\right) +  \sum_{n = 0}^{\infty}q^{n(n + 1)/2} \,\,\,(\text{by}\,\,\,\, \eqref{eq29}) \\
& \equiv \sum_{n = -\infty}^{\infty} q^{3n^{2}+n} + \sum_{n = 0}^{\infty}q^{n(n + 1)/2}.
\end{align*}

\noindent The series on the right-hand side is the same as the series expansion for $\sum\limits_{n = 0}^{\infty}c_{3}(n)q^{n}$  and so Theorem \ref{same0} is valid for $c_{11}(n)$, i.e.
\begin{equation*}
c_{11}(11n + 5) \equiv 0 \pmod{2},
\end{equation*}
\begin{equation*}
c_{11}(11n + 7) \equiv 0 \pmod{2}
\end{equation*}
and
\begin{equation*}
c_{11}(11n + 9) \equiv 0 \pmod{2}.
\end{equation*}

\noindent Let $c_{12}(n)$ denote the number of partitions of $n$ in which either
\begin{enumerate}
\item[(a)] all parts are distinct and greater than or equal to 2\\
 or
\item[(b)] the largest  repeated part $j$ appears exactly  three times  if $j \equiv 0  \pmod{2}$, and appears three or four times  if  $j \equiv 1 \pmod{2}$, all positive even integers  $< j$ appear exactly twice, all positive odd integers $< j$ appear two or three times, all even parts $>j$ are actually at least $2j + 2$ in part size and distinct, odd parts $> j$ are distinct and no odd part is equal to $2j+1$. 
\end{enumerate}
We have:
\begin{theorem}\label{part27}
For all $n\geq 0$,
$$c_{12}(n)
\equiv \begin{cases}
1\pmod{2}, & n = 6j^2 + 4j, j \in \mathbb{Z};\\\\
0\pmod{2}, & \text{otherwise}.
\end{cases}
$$
\end{theorem}

We have
\begin{align*}
\sum_{n = 0}^{\infty} c_{12}(n)q^{n} & = (-q^{2};q)_{\infty} + \sum_{n = 1}^{\infty} q^{1 + 1 + 2 + 2 + \cdots + (n - 1) + (n - 1) + n + n + n}(-q;q^{2})_{n}(-q^{2n+2};q)_{\infty} \\
& =  \sum_{n = 0}^{\infty} q^{n(n + 2)}(-q;q^{2})_{n}(-q^{2n+2};q)_{\infty} \\
&  \equiv \sum_{n = 0}^{\infty} q^{n(n+2)}(q;q^{2})_{n}(q^{2n+2};q^{2})_{\infty}(q^{2n + 3};q^{2})_{\infty})\\
& \equiv \sum_{n = 0}^{\infty} \frac{q^{n(n+2)}(q^{2n+2};q^{2})_{\infty}(q;q^{2})_{\infty}}{(1-q^{2n+1})}\\
& = (q;q^{2})_{\infty} \sum_{n = 0}^{\infty} \dfrac{q^{n(n+2)}(q^{2};q^{2})_{\infty}}{(q^{2};q^{2})_{n}(1-q^{2n+1})} \\
& = (q;q)_{\infty} \sum_{n = 0}^{\infty} \dfrac{q^{n(n+2)}}{(q^{2};q^{2})_{n}(1-q^{2n+1})} \\
& = (q;q)_{\infty} \sum_{n = 0}^{\infty} \dfrac{q^{n(n+2)}(q;q^{2})_{n}}{(q;q^{2})_{n}(q^{2};q^{2})_{n}(1-q^{2n+1})} \\
& \equiv (q;q)_{\infty} \sum_{n = 0}^{\infty} \dfrac{q^{n(n+2)}(-q;q^{2})_{n}}{(q;q^{2})_{n+1}(q^{2};q^{2})_{n}} \\
 & = (q;q)_{\infty} \sum_{n = 0}^{\infty} \dfrac{q^{n(n+2)}(-q;q^{2})_{n}}{(q;q)_{2n+1}} \\
& = \prod_{n=1}^{\infty} \left(1-q^{12n-2}\right)\left(1-q^{12n-10}\right)\left(1-q^{12n}\right) \,\,\,(\text{by}\,\,\,\, \eqref{eq50}) \\
& \equiv \sum_{n = -\infty}^{\infty} q^{6n^{2}+4n}.
\end{align*}
\noindent The following corollary is immediately noticeable.
\begin{corollary}
For all $n\geq 0$, we have,
$$c_{12}(2n + 1) \equiv 0 \pmod{2}.$$
\end{corollary}


\end{document}